\documentclass[english]{amsart}
\usepackage{a4wide}

\newtheorem{theorem}{Theorem}[section]
\newtheorem{lemma}[theorem]{Lemma}

\newtheorem{proposition}[theorem]{Proposition}
\newtheorem{corollary}[theorem]{Corollary}

\theoremstyle{definition}
\newtheorem{definition}[theorem]{Definition}

\newtheorem{remark}[theorem]{Remark}

\numberwithin{equation}{section}

\newcommand{\abs}[1]{\lvert#1\rvert}

\usepackage{graphicx}

\usepackage[fixlanguage]{babelbib}
\selectbiblanguage{english}

\usepackage[shortlabels]{enumitem}

\usepackage{amsfonts,amssymb,amsbsy,amsmath,mathrsfs,amsthm}

\usepackage{commath}

\usepackage{hyperref}
\hypersetup{pdfpagemode=UseNone, pdfstartview=FitH,
  colorlinks=true,urlcolor=red,linkcolor=blue,citecolor=black,
  pdftitle={Default Title, Modify},pdfauthor={Your Name},
  pdfkeywords={Modify keywords}}

\RequirePackage{babel}

\providecommand{\abs}[1]{ \lvert#1  \rvert}
\providecommand{\norm}[1]{ \lVert#1  \rVert}

\newcommand{\dx}{\, d x}

\newcommand{\dt}{\, d t}

\newcommand{\dla}{\, d \lambda}

\def\Xint#1{\mathchoice
   {\XXint\displaystyle\textstyle{#1}}%
   {\XXint\textstyle\scriptstyle{#1}}%
   {\XXint\scriptstyle\scriptscriptstyle{#1}}%
   {\XXint\scriptscriptstyle\scriptscriptstyle{#1}}%
   \!\int}
\def\XXint#1#2#3{{\setbox0=\hbox{$#1{#2#3}{\int}$}
     \vcenter{\hbox{$#2#3$}}\kern-.5\wd0}}

\def\dashint{\Xint-}

\DeclareMathOperator{\BMO}{\mathrm{BMO}}
\DeclareMathOperator{\PBMO}{\mathrm{PBMO}}

\usepackage{cite}
\makeatletter
\newcommand{\citecomment}[2][]{\citen{#2}#1\citevar}
\newcommand{\citeone}[1]{\citecomment{#1}}
\newcommand{\citetwo}[2][]{\citecomment[,~#1]{#2}}
\newcommand{\citevar}{\@ifnextchar\bgroup{;~\citeone}{\@ifnextchar[{;~\citetwo}{]}}}
\newcommand{\citefirst}{\@ifnextchar\bgroup{\citeone}{\@ifnextchar[{\citetwo}{]}}}

\makeatother

\usepackage{mathtools}
\makeatletter
\newcommand{\vast}{\bBigg@{3}}
\newcommand{\Vast}{\bBigg@{4}}
\newcommand{\vastl}{\mathopen\vast}
\newcommand{\vastr}{\mathclose\vast}
\makeatother

\begin{document}

\title{Parabolic John--Nirenberg spaces with time lag}

\author{Kim Myyryl\"ainen}
\address{Department of Mathematics, Aalto University, P.O. Box 11100, FI-00076 Aalto, Finland}
\email{kim.myyrylainen@aalto.fi}

\author{Dachun Yang}
\address{Laboratory of Mathematics and Complex Systems (Ministry of Education of China), 
School of Mathematical Sciences, Beijing Normal University, Beijing 100875, People's Republic of China}
\email{dcyang@bnu.edu.cn}

\thanks{
This project was supported by 
the National Key Research and Development Program of China (Grant No.\ 2020YFA0712900)
and
the National Natural Science Foundation of China
(Grant Nos.\ 12371093 and 12071197).
The first author was also supported by the Magnus Ehrnrooth Foundation.}

\subjclass[2020]{42B35}

\keywords{Parabolic John--Nirenberg space, John--Nirenberg inequality, parabolic geometry, time lag}

\begin{abstract}
We introduce a parabolic version of the so-called John--Nirenberg space that is a generalization of functions of parabolic bounded mean oscillation. Parabolic John--Nirenberg inequalities, which give weak type estimates for the oscillation of a function, are shown in the setting of the parabolic geometry with a time lag. Our arguments are based on a parabolic Calder\'{o}n--Zygmund decomposition and a good lambda estimate. Chaining arguments are applied to change the time lag in the parabolic John--Nirenberg inequality.

\end{abstract}

\maketitle

\section{Introduction}

Functions of bounded mean oscillation $\BMO$ and their generalization, the so-called John--Nirenberg space denoted by $JN_q$ with $1<q<\infty$, were introduced by John and Nirenberg in~\cite{john_original}.
Let $\Omega$ be a domain in $\mathbb{R}^n$.
A function $f\in L_{\mathrm{loc}}^1(\Omega)$ belongs to 
the John--Nirenberg space
$JN_q(\Omega)$, $1<q<\infty$, if 
\[
\sup \sum_{i=1}^\infty \lvert Q_i \rvert \inf_{c \in \mathbb{R}} \biggl( \dashint_{Q_i} \lvert f-c \rvert \dx \biggr)^q < \infty ,
\]
where the supremum is taken over countable collections $\{Q_i\}_{i\in\mathbb N}$ of pairwise disjoint subcubes of $\Omega$.
A particularly useful result for $\BMO$ is the John--Nirenberg lemma which gives an exponential decay estimate for the mean oscillation of a function in $\BMO$.
Moreover, the John--Nirenberg lemma for $JN_q$ gives a weak type estimate for the oscillation of a function in $JN_q$ implying that the John--Nirenberg space
is a subset of weak $L^q$.
The corresponding time-dependent theory 
was initiated by Moser in~\cite{moser1964,moser1967}
to study the regularity of parabolic partial differential equations,
and parabolic $\BMO$ was explicitly defined by Fabes and Garofalo in~\cite{fabesgarofalo}.
The proof of the parabolic John--Nirenberg lemma 
requires genuinely new ideas compared to the time independent case.
The theory of parabolic $\BMO$ 
has been further studied in~\cite{aimar,kinnunenSaariParabolicWeighted,KinnunenMyyryYang2022,localtoglobal,KinnunenMyyryYangZhu2022,KinnunenMyyry2023}.
See also~\cite{martin1994} for the one-dimensional case.

The purpose of this paper is to introduce and study a parabolic version of the John--Nirenberg space in the parabolic geometry with a time lag $0<\gamma<1$.
This geometry is motivated by certain parabolic nonlinear partial differential equations, see~\cite{aimar,kinnunenSaariParabolicWeighted,KinnunenMyyryYang2022,localtoglobal,KinnunenMyyryYangZhu2022,KinnunenMyyry2023}.
Let $\Omega_T = \Omega \times (0,T)$ be a space-time cylinder in $\mathbb{R}^{n+1}$.
A function $f\in L_{\mathrm{loc}}^1(\Omega_T)$ belongs to 
the parabolic John--Nirenberg space
$PJN_q^+(\Omega_T)$, $1<q<\infty$, if 
\[
\sup \sum_{i=1}^\infty \lvert R_i^+(\gamma) \rvert \inf_{c \in \mathbb{R}} \Biggl( \dashint_{R_i^+(\gamma)} (f-c)_+ \dx \dt + \dashint_{R_i^-(\gamma)} (f-c)_- \dx \dt \Biggr)^q < \infty ,
\]
where the supremum is taken over countable collections $\{R_i\}_{i\in\mathbb N}$ of pairwise disjoint 
space-time 
subrectangles of $\Omega_T$.
Here $R_i$ is a parabolic rectangle in the parabolic $p$-geometry with $1< p<\infty$ and $0<\gamma<1$ is the time lag, see Section~\ref{propertiesPJNq} for notation. 
In a similar way to the standard setting, this space is a generalization of parabolic $\BMO$ in a sense that parabolic $JN_q$ contains parabolic $\BMO$, and parabolic $\BMO$ is obtained as the limit of parabolic John--Nirenberg spaces as $q\to\infty$.
A one-sided John--Nirenberg space has been considered by Berkovits~\cite{berkovits2012} in the case $p=1$, but the geometry is not related to nonlinear parabolic partial differential equations.
For recent 
developments
in the theory of John--Nirenberg spaces in the standard setting, we refer to~\cite{myyry2022, kinnunenmyyry2023a,takala2022,tao2021survey,kortetakala2023,dafni2018,tao2022,tao2019,milman2021,marolasaari,molla2022}.

Our main results are parabolic John--Nirenberg inequalities for 
parabolic $JN_q$.
The main challenges are the time lag between 
the two mean oscillations
in the definition and the parabolic geometry.
The proof of the parabolic John--Nirenberg inequality in Theorem~\ref{local_pJN} is based on a parabolic Calder\'{o}n--Zygmund decomposition
and a good lambda inequality.
Theorem~\ref{global_pJN} shows that by applying suitable parabolic chaining arguments we may improve the obtained John--Nirenberg lemma by changing the time lag.
This allows us to conclude that the parabolic John--Nirenberg space 
is independent of the size of the time lag,
see Corollary~\ref{equivPJNq}.
For more about chaining techniques in the parabolic geometry, see~\cite{localtoglobal,forward-in-time,KinnunenMyyry2023}.

\section{Definition and properties of parabolic John--Nirenberg spaces}
\label{propertiesPJNq}

The underlying space throughout is $\mathbb{R}^{n+1}=\{(x,t):x=(x_1,\dots,x_n)\in\mathbb R^n,t\in\mathbb R\}$.
Unless otherwise stated, constants are positive and the dependencies on parameters are indicated in the brackets.
The Lebesgue measure of a measurable subset $A$ of $\mathbb{R}^{n+1}$ is denoted by $\lvert A\rvert$.
A cube $Q$ is a bounded interval in $\mathbb R^n$, with edges parallel to the coordinate axes and equally long, that is,
$Q=Q(x,L)=\{y \in \mathbb R^n: \lvert y_i-x_i\rvert \leq L,\,i=1,\dots,n\}$
with $x\in\mathbb R^n$ and $L>0$. 
The point $x$ is the center of the cube and $L$ is the edge length of the cube. 
Instead of Euclidean cubes, we work with the following collection of parabolic rectangles in $\mathbb{R}^{n+1}$.

\begin{definition}\label{def_parrect}
Let $1<p<\infty$, $x\in\mathbb R^n$, $L>0$ and $t \in \mathbb{R}$.
A parabolic rectangle centered at $(x,t)$ with edge length $L$ is
\[
R = R(x,t,L) = Q(x,L) \times (t-L^p, t+L^p)
\]
and its upper and its lower parts are
\[
R^+(\gamma) = Q(x,L) \times (t+\gamma L^p, t+L^p) 
\quad\text{and}\quad
R^-(\gamma) = Q(x,L) \times (t - L^p, t - \gamma L^p) ,
\]
where $-1 < \gamma < 1$ is called the time lag.
\end{definition}

Note that $R^-(\gamma)$ is the reflection of $R^+(\gamma)$ with respect to the time slice $\mathbb{R}^n \times \{t\}$.
The spatial edge length of a parabolic rectangle $R$ is denoted by $l_x(R)=L$ and the time length by $l_t(R)=2L^p$.
For short, we write $R^\pm$ for $R^{\pm}(0)$.
The top of a rectangle $R = R(x,t,L)$ is $Q(x,L) \times\{t+L^p\}$
and the bottom is $Q(x,L) \times\{t-L^p\}$.
The $\lambda$-dilate of $R$  with $\lambda>0$ is denoted by $\lambda R = R(x,t,\lambda L)$.

The integral average of $f \in L^1(A)$ in measurable set $A\subset\mathbb{R}^{n+1}$, with $0<|A|<\infty$, is denoted by
\[
f_A = \dashint_A f \dx \dt = \frac{1}{\lvert A\rvert} \int_A f(x,t)\dx\dt .
\]
The positive and the negative parts of a function $f$ are denoted by
\[
f_+ = \max\{ f, 0 \} 
\quad\text{and}\quad
f_- = - \min\{ f , 0 \} .
\]
Let $\Omega \subset \mathbb{R}^{n}$ be an open set and $T>0$.
A space-time cylinder is denoted by $\Omega_T=\Omega\times(0,T)$.
It is possible to consider space-time cylinders $\Omega\times(t_1,t_2)$ with $t_1<t_2$, but we focus on $\Omega_T$.

We recall the definition of the parabolic $\BMO$.
The differentials $\dx \dt$ in integrals are omitted in the sequel.

\begin{definition}
Let $\Omega \subset \mathbb{R}^{n}$ be a domain, $T>0$, $ 0 \leq \gamma < 1$ and $0< r < \infty$.
A function $f \in L_{\mathrm{loc}}^r(\Omega_T)$ belongs to 
$\PBMO_{\gamma,r}^{+}(\Omega_T)$ if
\[
\norm{f}_{\PBMO_{\gamma,r}^{+}(\Omega_T)} 
= \sup_{R \subset \Omega_T} \inf_{c \in \mathbb{R}} \Biggl( \dashint_{R^+(\gamma)} (f-c)_+^r 
+ \dashint_{R^-(\gamma)} (f-c)_-^r \Biggr)^\frac{1}{r} < \infty.
\]
If the condition above holds with the time axis reversed, then $f \in \PBMO_{\gamma,r}^{-}(\Omega_T)$.
\end{definition}

This section discusses basic properties of the parabolic John--Nirenberg space.
We begin with the definition.

\begin{definition}
Let $\Omega \subset \mathbb{R}^{n}$ be a domain, $T>0$, $ 0 \leq \gamma < 1$, $1 < q < \infty$ and $0 < r < q$.
A function $f \in L_{\mathrm{loc}}^r(\Omega_T)$ belongs to 
$PJN_{q,\gamma,r}^{+}(\Omega_T)$ if 
\[
\norm{f}_{PJN_{q,\gamma,r}^{+}(\Omega_T)}^q = \sup \sum_{i=1}^\infty \lvert R_i^+(\gamma) \rvert \inf_{c \in \mathbb{R}} \Biggl( \dashint_{R_i^+(\gamma)} (f-c)_+^r  + \dashint_{R_i^-(\gamma)} (f-c)_-^r \Biggr)^\frac{q}{r} < \infty ,
\]
where the supremum is taken over countable collections $\{R_i\}_{i\in\mathbb N}$ of pairwise disjoint parabolic subrectangles of $\Omega_T$.
If the condition above holds with the time axis reversed, then $f \in PJN_{q,\gamma,r}^{-}(\Omega_T)$.
\end{definition}

We shall write $PJN_{q}^{+}$ and $\norm{f}$ whenever parameters are clear from the context or are not of importance.
Observe that  $f \in L_{\mathrm{loc}}^r(\Omega_T)$ belongs to $PJN_{q,\gamma,r}^{+}(\Omega_T)$ if and only if
for every collection $\{R_i\}_{i\in\mathbb{N}}$ of parabolic rectangles there exist constants $c_i\in\mathbb R$, that may depend on $R_i$, with
\[
\sum_{i=1}^\infty \lvert R_i^+(\gamma) \rvert \biggl( \dashint_{R_i^+(\gamma)} (f-c_i)_+^r \biggr)^\frac{q}{r}
\le M
\quad\text{and}\quad
\sum_{i=1}^\infty \lvert R_i^-(\gamma) \rvert \biggl( \dashint_{R_i^-(\gamma)} (f-c_i)_-^r \biggr)^\frac{q}{r}
\le M,
\]
where  $M\in\mathbb R$ is a constant that is independent of $\{R_i\}_{i\in\mathbb{N}}$.

The next lemma shows that for every parabolic rectangle $R$, there exists a constant $c_R$, depending on $R$, for which the infimum in the definition above is attained.
In the sequel, this minimal constant is denoted by $c_R$.
The proof is analogous to that of the parabolic $\BMO$ in~\cite{KinnunenMyyryYang2022}, and thus is omitted here.

\begin{lemma}
\label{PJNp_constant}
Let $\Omega_T \subset \mathbb{R}^{n+1}$ be a space-time cylinder, $ 0 \leq \gamma < 1$, $1< q < \infty$ and $0<r<q$.
Assume that $f\in PJN_{q,\gamma,r}^{+}(\Omega_T)$.
Then for every parabolic rectangle $R\subset\Omega_T$, there exists a constant $c_R\in\mathbb R$, that may depend on $R$, such that 
\[
\dashint_{R^+(\gamma)} (f-c_R)_+^r  + \dashint_{R^-(\gamma)} (f-c_R)_-^r = \inf_{c \in \mathbb{R}} \Biggl( \dashint_{R^+(\gamma)} (f-c)_+^r + \dashint_{R^-(\gamma)} (f-c)_-^r \Biggr) .
\]
In particular, it holds that
\[
\sup \sum_{i=1}^\infty \lvert R_i^+(\gamma) \rvert \Biggl( \dashint_{R_i^+(\gamma)} (f-c_{R_i})_+^r + \dashint_{R_i^-(\gamma)} (f-c_{R_i})_-^r \Biggr)^\frac{q}{r} = \norm{f}_{PJN_{q,\gamma,r}^{+}(\Omega_T)}^q .
\]
\end{lemma}

We list some properties of the parabolic John--Nirenberg space below.
In particular, $PJN_{q}^{+}$ is closed under addition and scaling by a positive constant. 
On the other hand, multiplication by negative constants reverses the time direction.
The proof is similar to that of the parabolic $\BMO$ in~\cite{KinnunenMyyryYang2022}, and thus is omitted here.

\begin{lemma}
\label{props_PJNp}
Let $ 0 \leq \gamma < 1$, $1< q < \infty$ and $0<r<q$.
Assume that $f$ and $ g$ belong to $PJN_{q,\gamma,r}^{+}(\Omega_T)$ and
let $PJN_{q}^\pm=PJN_{q,\gamma,r}^\pm(\Omega_T)$.
Then the following properties hold.
\begin{enumerate}[\normalfont(i), parsep=5pt, topsep=5pt]
    \item $\begin{aligned}[t]
        \norm{f+a}_{PJN_{q}^{+}} = \norm{f}_{PJN_{q}^{+}},\, a \in \mathbb{R}.
    \end{aligned}$
    \item $\begin{aligned}[t]
        \norm{f+g}_{PJN_{q}^{+}} \leq \max\{ 2^{\frac{1}{r}-1} , 2^{1-\frac{1}{r}} \} \Bigl( \norm{f}_{PJN_{q}^{+}} + \norm{g}_{PJN_{q}^{+}} \Bigr).
    \end{aligned}$
    \item $\begin{aligned}[t]
        \norm{af}_{PJN_{q}^{+}} = \begin{cases} 
    \displaystyle
      a \norm{f}_{PJN_{q}^{+}}, &a \geq 0, \\
      \displaystyle
      -a \norm{f}_{PJN_{q}^{-}}, & a < 0.
   \end{cases}
    \end{aligned}$
    \item $\begin{aligned}
        \norm{\max\{f,g\}}_{PJN_{q}^{+}} &\leq \max\{1, 2^{\frac{1}{r} -1 }\} \Bigl( \norm{f}_{PJN_{q}^{+}} + \norm{g}_{PJN_{q}^{+}} \Bigr), 
        \\
        \norm{\min\{f,g\}}_{PJN_{q}^{+}} &\leq \max\{1, 2^{\frac{1}{r} -1 }\} \Bigl( \norm{f}_{PJN_{q}^{+}} + \norm{g}_{PJN_{q}^{+}} \Bigr).
    \end{aligned}$
\end{enumerate}
\end{lemma}

\begin{remark}
For $1\leq r<q$, the constants in (ii) and (iv) can be avoided by considering the norm
\[
\sup \left( \sum_{i=1}^\infty \lvert R_i^+(\gamma) \rvert \inf_{c \in \mathbb{R}}\Biggl[ \biggl( \dashint_{R_i^+(\gamma)} (f-c)_+^r \biggr)^\frac{1}{r} 
+ \biggl( \dashint_{R_i^-(\gamma)} (f-c)_-^r \biggr)^\frac{1}{r}\Biggr]^q \right)^\frac{1}{q} ,
\]
which is an equivalent norm with our definition. 
However, the current definition will be convenient in the proof of the John--Nirenberg lemma below.
\end{remark}

Properties (i) and (iv) of Lemma~\ref{props_PJNp} imply that every $PJN_{q,\gamma,r}^{+}$ function can be approximated pointwise by bounded $PJN_{q,\gamma,r}^{+}$ functions.

\begin{remark}
If $f\in \PBMO_{\gamma,r}^{+}(\Omega_T)$ with $\lvert \Omega_T \rvert < \infty$, then $f \in PJN_{q,\gamma,r}^{+}(\Omega_T)$ for every $1<q<\infty$.
More precisely, we have
\[
\norm{f}_{PJN_{q,\gamma,r}^{+}(\Omega_T)} \leq \lvert \Omega_T \rvert^\frac{1}{q} \norm{f}_{\PBMO_{\gamma,r}^{+}(\Omega_T)},
\]
which follows from
\begin{align*}
\sum_{i=1}^\infty \lvert R_i^+(\gamma) \rvert \inf_{c \in \mathbb{R}} \Biggl( \dashint_{R_i^+(\gamma)} (f-c)_+^r  + \dashint_{R_i^-(\gamma)} (f-c)_-^r \Biggr)^\frac{q}{r} 
&\leq 
\norm{f}_{\PBMO_{\gamma,r}^{+}(\Omega_T)}^q  \sum_{i=1}^\infty \lvert R_i^+(\gamma) \rvert \\
&\leq \norm{f}_{\PBMO_{\gamma,r}^{+}(\Omega_T)}^q \lvert \Omega_T \rvert .
\end{align*}
\end{remark}

The parabolic John--Nirenberg space is a generalization of the parabolic BMO in the sense that a function is in PBMO$^+_{\gamma,r}$ if and only if the $PJN_{q,\gamma,r}^{+}$ norm of the function is uniformly bounded as $q$ tends to infinity.

\begin{proposition}
Let $\Omega_T\subset\mathbb{R}^{n+1}$ be of finite measure, $ 0 \leq \gamma < 1$
and $0<r<\infty$. Assume that $f \in L_{\mathrm{loc}}^r(\Omega_T)$.
Then
\[
\lim_{q \to \infty} \norm{f}_{PJN_{q,\gamma,r}^{+}(\Omega_T)} = \norm{f}_{\PBMO_{\gamma,r}^{+}(\Omega_T)}.
\]
\end{proposition}

\begin{proof}
We may assume that $\lvert \Omega_T \rvert >0$.
Let $\{R_i\}_i$ be a collection of pairwise disjoint parabolic subrectangles of $\Omega_T$. 
Recall that if $\lvert \Omega_T \rvert < \infty$, then $\norm{\cdot}_{L^q(\Omega_T)} \to \norm{\cdot}_{L^\infty(\Omega_T)}$ as $q\to\infty$.
It follows that
\begin{align*}
& \left(  \sum_{i=1}^\infty \lvert R_i^+(\gamma) \rvert \inf_{c \in \mathbb{R}} \Biggl( \dashint_{R_i^+(\gamma)} (f-c)_+^r + \dashint_{R_i^-(\gamma)} (f-c)_-^r \Biggr)^\frac{q}{r} \right)^\frac{1}{q} \\
&\qquad=
\left( \int_{\Omega_T} \vastl( \sum_{i=1}^\infty \chi_{R_i^+(\gamma)}(x,t) \inf_{c \in \mathbb{R}} \Biggl( \dashint_{R_i^+(\gamma)} (f-c)_+^r + \dashint_{R_i^-(\gamma)} (f-c)_-^r \Biggr)^\frac{1}{r} \vastr)^q \dx\dt \right)^\frac{1}{q} \\
&\qquad\to
\sup_{(x,t)\in \Omega_T} 
\sum_{i=1}^\infty \chi_{R_i^+(\gamma)}(x,t) \inf_{c \in \mathbb{R}}  \Biggl( \dashint_{R_i^+(\gamma)} (f-c)_+^r + \dashint_{R_i^-(\gamma)} (f-c)_-^r \Biggr)^\frac{1}{r} \\
&\qquad=
\sup_{i} \inf_{c \in \mathbb{R}} \Biggl( \dashint_{R_i^+(\gamma)} (f-c)_+^r + \dashint_{R_i^-(\gamma)} (f-c)_-^r \Biggr)^\frac{1}{r}
\end{align*}
as $q\to\infty$.
Hence, we have
\begin{align*}
& \sup_{\{R_i\}_i} \lim_{q\to\infty} \left(  \sum_{i=1}^\infty \lvert R_i^+(\gamma) \rvert \inf_{c \in \mathbb{R}} \Biggl( \dashint_{R_i^+(\gamma)} (f-c)_+^r + \dashint_{R_i^-(\gamma)} (f-c)_-^r \Biggr)^\frac{q}{r} \right)^\frac{1}{q} \\
&\qquad= \sup_{\{R_i\}_i} \sup_{i} \inf_{c \in \mathbb{R}} \Biggl( \dashint_{R_i^+(\gamma)} (f-c)_+^r + \dashint_{R_i^-(\gamma)} (f-c)_-^r \Biggr)^\frac{1}{r} \\
&\qquad= \sup_{R\subset \Omega_T} \inf_{c \in \mathbb{R}} \Biggl( \dashint_{R^+(\gamma)} (f-c)_+^r + \dashint_{R^-(\gamma)} (f-c)_-^r \Biggr)^\frac{1}{r} \\
&\qquad= \norm{f}_{\PBMO_{\gamma,r}^{+}(\Omega_T)} .
\end{align*}
We can interchange the order of taking the supremum and the limit since 
\[
\left( \dashint_{\Omega_T} \vastl( \sum_{i=1}^\infty \chi_{R_i^+(\gamma)}(x,t) \inf_{c \in \mathbb{R}} \Biggl( \dashint_{R_i^+(\gamma)} (f-c)_+^r + \dashint_{R_i^-(\gamma)} (f-c)_-^r \Biggr)^\frac{1}{r} \vastr)^q \dx\dt \right)^\frac{1}{q}
\]
is an increasing function of $q$ which can be seen by H\"{o}lder's inequality.
Thus, we conclude that
\[
\lim_{q\to\infty} \norm{f}_{PJN_{q,\gamma,r}^{+}(\Omega_T)} = \norm{f}_{\PBMO_{\gamma,r}^{+}(\Omega_T)} .
\qedhere
\]
\end{proof}

\section{Parabolic John--Nirenberg inequality}

This section discusses a John--Nirenberg inequality for parabolic John--Nirenberg spaces.
The argument is based on a parabolic Calder\'{o}n--Zygmund decomposition and applies a parabolic sharp maximal function to obtain a good lambda estimate.
For short, we suppress the variables $(x,t)$ in the notation and, for example, write
\[
R^{+}(\alpha)\cap\{(f-c_{R})_+ > \lambda\}
=\{(x,t)\in R^{+}(\alpha):(f(x,t)-c_{R})_+ > \lambda\} 
\]
in the sequel.

\begin{theorem}
\label{local_pJN}
Let $R \subset \mathbb{R}^{n+1}$ be a parabolic rectangle, $0 \leq \gamma < 1$, $\gamma < \alpha < 1$, $1 < q < \infty$ and $0 < r \leq 1$.
Assume that $f \in PJN_{q,\gamma,r}^{+}(R)$.
Then there exist constants $c_R\in\mathbb R$ and $C=C(n,p,q,r,\gamma,\alpha)$ such that
\[
\lvert  R^{+}(\alpha) \cap \{ (f-c_{R})_+ > \lambda \} \rvert \leq C  \frac{\norm{f}_{PJN_{q,\gamma,r}^{+}(R)}^q}{\lambda^q} 
\]
and
\[
\lvert R^{-}(\alpha) \cap \{ (f-c_{R})_- > \lambda \} \rvert \leq C \frac{\norm{f}_{PJN_{q,\gamma,r}^{+}(R)}^q}{\lambda^q} 
\]
for every $\lambda >0$.
\end{theorem}

\begin{proof}
The beginning of the proof is similar to the John--Nirenberg lemma for parabolic BMO in~\cite{KinnunenMyyryYang2022}.
However,
the rest of the proof uses several new ideas when dealing with
weak type estimates for parabolic $JN_q$
compared to the exponential decay estimate for parabolic $\BMO$.
Let $R_0=R=R(x_0,t_0,L) = Q(x_0,L) \times (t_0-L^p, t_0+L^p)$.
By considering the function $f(x+x_0,t+t_0)$, we may assume that the center of $R_0$ is the origin, that is, $R_0 = Q(0,L) \times (-L^p, L^p)$.
By (i) of Lemma~\ref{props_PJNp}, we may assume that $c_{R_0} = 0$.
We note that it is sufficient to prove the first inequality of the theorem since
the second inequality follows from the first one by applying it to the function $-f(x,-t)$.

Let
$\norm{f} = \norm{f}_{PJN_{q,\gamma,r}^{+}(R_0)}$.
We claim that
\[
\lvert R^{+}_0(\alpha)\cap\{f_+ > \lambda \} \rvert 
\leq C  \frac{\norm{f}^q}{\lambda^q}
\]
for every $\lambda>0$.
Let $m$ be the smallest integer with
$3 + \alpha \leq 2^{pm+1} (\alpha - \gamma)$, 
that is,
\[
\frac{1}{p} \log_2 \biggl( \frac{3+ \alpha}{2 (\alpha - \gamma)} \biggr) \leq m < \frac{1}{p} \log_2 \biggl( \frac{3+ \alpha}{2 (\alpha - \gamma)} \biggr) + 1 .
\]
Let $S^+_0 = R^+_0(\alpha) = Q(0,L) \times (\alpha L^p,L^p) $. The time length of $S^+_0$ is $l_t(S^+_0) = (1-\alpha)L^p$.
We partition $S^+_0$ by dividing each spatial edge $[-L,L]$ into $2^m$ equally long intervals.
If
\[
\frac{l_t(S_{0}^+)}{\lfloor 2^{pm} \rfloor} < \frac{(1-\alpha)L^p}{2^{pm}},
\]
we divide the time interval of $S^+_0$ into $\lfloor 2^{pm} \rfloor$ equally long intervals. 
Otherwise, we divide the time interval of $S^+_0$ into $\lceil 2^{pm} \rceil$ equally long intervals.
Here $\lfloor \cdot \rfloor$ and $\lceil \cdot \rceil$ are the floor and ceiling functions, respectively.
We obtain subrectangles $S^+_1$ of  $S^+_0$ with spatial edge length 
$l_x(S^+_1)=l_x(S^+_0)/2^m = L / 2^m$ and time length either 
\[
l_t(S^+_1)=\frac{l_t(S^+_0)}{\lfloor 2^{pm} \rfloor} 
=\frac{(1-\alpha)L^p}{\lfloor 2^{pm} \rfloor} 
\quad\text{or}\quad
l_t(S^+_1)=\frac{(1-\alpha)L^p}{\lceil 2^{pm} \rceil}.
\]
For every $S^+_1$, there exists a unique rectangle $R_1$ with spatial edge length $l_x = L / 2^{m}$ and time length $l_t = 2 L^p / 2^{mp}$
such that $R_1$ has the same top as $S^+_1$.
We select those rectangles $S^+_1$ for which $\lambda < c_{R_1}$ and denote the obtained collection by $\{ S^+_{1,j} \}_j$.
If $\lambda \geq c_{R_1}$, we subdivide $S^+_1$ in the same manner as above
and select all those subrectangles $S^+_2$ for which $\lambda < c_{R_2}$ to obtain family $\{ S^+_{2,j} \}_j$.
We continue this selection process recursively.
At the $i$th step, we partition unselected rectangles $S^+_{i-1}$ by dividing each spatial edge into $2^m$ equally long intervals. 
If 
\begin{equation}
\label{JNproof_eq1}
\frac{l_t(S_{i-1}^+)}{\lfloor 2^{pm} \rfloor} < \frac{(1-\alpha)L^p}{2^{pmi}},
\end{equation}
we divide the time interval of $S^+_{i-1}$ into $\lfloor 2^{pm} \rfloor$ equally long intervals. 
If
\begin{equation}
\label{JNproof_eq2}
\frac{l_t(S_{i-1}^+)}{\lfloor 2^{pm} \rfloor} \geq \frac{(1-\alpha)L^p}{2^{pmi}},
\end{equation}
we divide the time interval of $S^+_{i-1}$ into $\lceil 2^{pm} \rceil$ equally long intervals.
We obtain subrectangles $S^+_i$. 
For every $S^+_i$, there exists a unique rectangle $R_i$ with spatial edge length $l_x = L / 2^{mi}$ and time length $l_t = 2 L^p / 2^{pmi}$
such that $R_i$ has the same top as $S^+_i$.
Select those $S^+_i$ for which $\lambda < c_{R_i}$ and denote the obtained collection by $\{ S^+_{i,j} \}_j$.
If $\lambda \geq c_{R_i}$, we continue the selection process in $S^+_i$.
In this manner we obtain a collection $\{S^+_{i,j} \}_{i,j}$ of pairwise disjoint rectangles.

Observe that if \eqref{JNproof_eq1} holds, then we have
\[
l_t(S_i^+) = \frac{l_t(S^+_{i-1})}{\lfloor 2^{pm} \rfloor} < \frac{(1-\alpha)L^p}{2^{pmi}}.
\]
On the other hand, if \eqref{JNproof_eq2} holds, then
\[
l_t(S_i^+) = \frac{l_t(S^+_{i-1})}{\lceil 2^{pm} \rceil} \leq \frac{l_t(S^+_{i-1})}{2^{pm}} \leq \dots \leq \frac{(1-\alpha)L^p}{2^{pmi}} .
\]
This gives an upper bound 
\[
l_t(S_i^+) \leq \frac{(1-\alpha)L^p}{2^{pmi}}
\]
for every $S_i^+$.

Suppose that \eqref{JNproof_eq2} is satisfied at the $i$th step.
Then we have a lower bound for the time length of $S_i^+$, because
\[
l_t(S^+_i) = \frac{l_t(S_{i-1}^+)}{\lceil 2^{pm} \rceil} \geq \frac{\lfloor 2^{pm} \rfloor}{\lceil 2^{pm} \rceil} \frac{(1-\alpha)L^p}{2^{pmi}} \geq \frac{1}{2} \frac{(1-\alpha)L^p}{2^{pmi}} .
\]
On the other hand, if \eqref{JNproof_eq1} is satisfied, then
\[
l_t(S^+_i) = \frac{l_t(S_{i-1}^+)}{\lfloor 2^{pm} \rfloor} \geq \frac{l_t(S_{i-1}^+)}{ 2^{pm}}.
\]
In this case, \eqref{JNproof_eq2} has been satisfied at an earlier step $i'$ with $i'< i$.
We obtain
\[
l_t(S^+_i) \geq \frac{l_t(S_{i-1}^+)}{ 2^{pm}} \geq \dots \geq \frac{l_t(S_{i'}^+)}{ 2^{pm(i-i')}} \geq \frac{1}{2} \frac{(1-\alpha)L^p}{ 2^{pmi}}
\]
by using the lower bound for $S_{i'}^+$.
Thus, we have 
\[
\frac{1}{2} \frac{(1-\alpha)L^p}{2^{pmi}} \leq l_t(S^+_i) \leq \frac{(1-\alpha)L^p}{2^{pmi}}
\]
for every $S^+_i$.
By using the lower bound for the time length of  $S^+_i$ and the choice of $m$, we observe that
\begin{align*}
l_t(R_i) - l_t(S^+_i) 
&\leq \frac{2 L^p}{2^{pmi}} - \frac{1}{2} \frac{(1-\alpha)L^p}{2^{pmi}} 
= \frac{1}{2} \frac{L^p}{2^{pmi}} (3+ \alpha) \\
&\leq \frac{(\alpha-\gamma) L^p}{2^{pm(i-1)}} 
= \frac{(1-\gamma)L^p}{2^{pm(i-1)}} - \frac{(1-\alpha)L^p}{2^{pm(i-1)}} \\
&\leq l_t(R^+_{i-1}(\gamma)) - l_t(S^+_{i-1}) .
\end{align*}
This implies
\begin{equation}
\label{subset}
R_{i} \subset R^+_{i-1}(\gamma)
\end{equation}
for a fixed rectangle $S^+_{i-1}$ and for every subrectangle $S^+_{i} \subset S^+_{i-1}$.
By the construction of the subrectangles $S^+_i$, we have
\begin{equation}
\label{measure1}
2^{nm} \lfloor 2^{pm} \rfloor \lvert S^+_i \rvert 
\leq \lvert S^+_{i-1} \rvert \leq 2^{nm} \lceil 2^{pm} \rceil \lvert S^+_i \rvert 
\end{equation}
and
\begin{equation}
\label{measure2}
\frac{1}{2} \frac{1-\alpha}{1-\gamma} \lvert R^+_i(\gamma) \rvert 
\leq \lvert S^+_i \rvert \leq \frac{1-\alpha}{1-\gamma} \lvert R^+_i(\gamma) \rvert .
\end{equation}

We have a collection $\{ S^+_{i,j} \}_{i,j}$ of pairwise disjoint rectangles. 
However, the rectangles in the corresponding collection $\{ S^-_{i,j} \}_{i,j}$ may overlap. 
Thus, we replace it by a subfamily $\{ \widetilde{S}^-_{i,j} \}_{i,j}$ of pairwise disjoint rectangles, which is constructed in the following way.
At the first step, choose $\{ S^-_{1,j} \}_{j}$ and denote it by $\{ \widetilde{S}^-_{1,j} \}_j$. 
Then consider the collection $\{ S^-_{2,j} \}_{j}$ where each $S^-_{2,j}$ either intersects some $\widetilde{S}^-_{1,j}$ or does not intersect any $\widetilde{S}^-_{1,j}$. 
Select the rectangles $S^-_{2,j}$ that do not intersect any $\widetilde{S}^-_{1,j}$, and denote the obtained collection by $\{ \widetilde{S}^-_{2,j} \}_j$.
At the $i$th step, choose those $S^-_{i,j}$ that do not intersect any previously selected $\widetilde{S}^-_{i',j}$, $i' < i$.
Hence, we obtain a collection $\{ \widetilde{S}^-_{i,j} \}_{i,j}$ of pairwise disjoint rectangles.
Observe that for every $S^-_{i,j}$ there exists $\widetilde{S}^-_{i',j}$ with $i' \leq i$ such that
\begin{equation}
\label{plussubset}
\text{pr}_x(S^-_{i,j}) \subset \text{pr}_x(\widetilde{S}^-_{i',j}) \quad \text{and} \quad \text{pr}_t(S^-_{i,j}) \subset 3 \text{pr}_t(\widetilde{S}^-_{i',j}) .
\end{equation}
Here pr$_x$ denotes the projection to $\mathbb R^n$ and pr$_t$ denotes the projection to the time axis.

Rename $\{ S^+_{i,j} \}_{i,j}$ and $\{ \widetilde{S}^-_{i,j} \}_{i,j}$ as $\{ S^+_{i} \}_{i}$ and $\{ \widetilde{S}^-_{j} \}_j$, respectively.
Let $S(\lambda) = \bigcup_i S^+_i$.
Note that $S^+_i$ is spatially contained in $S^-_i$, that is, $\text{pr}_x S^+_i\subset \text{pr}_x S^-_i$.
In the time direction, we have
\begin{equation}
\label{minusplussubset}
\text{pr}_t(S^+_i) \subset \text{pr}_t(R_i) 
\subset \frac{7 + \alpha}{1-\alpha} \text{pr}_t(S^-_i) ,
\end{equation}
because
\[
\biggl( \frac{7 + \alpha}{1-\alpha} + 1 \biggr) \frac{l_t(S^-_i)}{2} \geq  \frac{8}{1-\alpha}  \frac{(1-\alpha)L^p}{2^{pmi+2}} = \frac{2L^p}{2^{pmi}} = l_t(R_i) .
\]
Therefore, by~\eqref{plussubset} and~\eqref{minusplussubset}, it holds that
\begin{equation}
\label{start}
\lvert S(\lambda) \rvert = \Big\lvert \bigcup_i S^+_i \Big\rvert \leq c_1 \sum_j \lvert \widetilde{S}^-_j \rvert 
\quad\text{with}\quad
c_1 = 3\frac{7+\alpha}{1-\alpha}.
\end{equation}

Let $\lambda > \delta > 0$ and consider the collection $\{S^+_k\}_k$ for $\delta$.
Define $\mathcal{J}_k = \{ j \in \mathbb{N}: \widetilde{S}^+_j \subset S^+_k \}$.
Since each $S^+_i$ is contained in some $S^+_k$, we get the partition
\[
\bigcup_j \widetilde{S}^-_j = \bigcup_k \bigcup_{j\in \mathcal{J}_k} \widetilde{S}^-_j .
\]
Fix $k \in \mathbb{N}$. We have $\widetilde{S}^+_j \subset S^+_k$ for every $j \in \mathcal{J}_k$, where $S^+_k$ was obtained by subdividing a previous $S^+_{k^-}$ for which $a_{R_{k^-}} \leq \delta$.
Hence, it holds that $\widetilde{S}^-_j \subset R_k$ for every $j \in \mathcal{J}_k$.
From \eqref{subset}, it follows that $\widetilde{S}^-_j \subset R^+_{k^-}(\gamma)$ for every $j \in \mathcal{J}_k$.
We have
\begin{align*}
\lambda^r &< c_{R_j}^r \leq \dashint_{\widetilde{S}^-_j} (c_{R_j} - f)^r_+  + \dashint_{\widetilde{S}^-_j} f^r_+ \\
&\leq \dashint_{\widetilde{S}^-_j} (f - c_{R_j})^r_- + \dashint_{\widetilde{S}^-_j} (f - a_{R_{k^-}})^r_+ + \delta^r \\
&= I_j + II_j + \delta^r
\end{align*}
for every $j \in \mathcal{J}_k$.
Let $\mathcal{M}_k = \{j\in \mathcal{J}_k: I_j \leq II_j \}$ and
\[
\mathcal{K} = \biggl\{ k \in \mathbb{N}: \lvert R^+_{k^-}(\gamma) \rvert \leq A \Bigl\lvert \bigcup_{j \in \mathcal{M}_k} \widetilde{S}^-_j \Bigr\rvert  \biggr\} ,
\]
where $A>1$ will be specified later. 
Then for $j \in \mathcal{M}_k$ we have
\begin{align*}
\lambda^r - \delta^r < 2 \dashint_{\widetilde{S}^-_j} (f - a_{R_{k^-}})^r_+ .
\end{align*}
By summing over $j \in \mathcal{M}_k$, we obtain
\begin{align*}
(\lambda^r - \delta^r) \sum_{j \in \mathcal{M}_k} \lvert \widetilde{S}^-_j \rvert < \sum_{j \in \mathcal{M}_k} 2 \int_{\widetilde{S}^-_j} (f - a_{R_{k^-}})^r_+ \leq 2 \int_{R^+_{k^-}(\gamma)} (f - a_{R_{k^-}})^r_+ .
\end{align*}
Thus, if $k \in \mathcal{K}$, we observe that
\begin{align*}
\frac{1}{A} (\lambda^r - \delta^r)  \leq \frac{\lambda^r - \delta^r}{\lvert R^+_{k^-}(\gamma) \rvert} \sum_{j \in \mathcal{M}_k} \lvert \widetilde{S}^-_j \rvert < 2 \dashint_{R^+_{k^-}(\gamma)} (f - a_{R_{k^-}})^r_+ \leq 2 M^\sharp f(x,t)
\end{align*}
for every $(x,t) \in R_{k^-}$, where $M^\sharp f$ is the parabolic sharp maximal function defined by
\[
M^\sharp f(x,t) = \sup_{\substack{R \ni (x,t) \\ R \subset R_0}} \Biggl( \dashint_{R^+(\gamma)} (f - c_R)^r_+  + \dashint_{R^-(\gamma)} (f - c_R)^r_- \Biggr) .
\]
It holds that
\[
M^\sharp f(x,t) > \frac{1}{2A} (\lambda^r - \delta^r)
\]
for every $(x,t) \in \bigcup_{k \in \mathcal{K}} R_{k^-}$.
Thus, we obtain
\begin{align*}
\Bigl\lvert \bigcup_{k \in \mathcal{K}} \bigcup_{j\in \mathcal{M}_k} \widetilde{S}^-_j \Bigr\rvert \leq \Bigl\lvert \bigcup_{k \in \mathcal{K}} R^+_{k^-}(\gamma) \Bigr\rvert \leq \Bigl\lvert \bigcup_{k \in \mathcal{K}} R_{k^-} \Bigr\rvert \leq \bigr\lvert R_0 \cap \{ M^\sharp f > \tfrac{1}{2A} (\lambda^r - \delta^r) \} \bigl\rvert .
\end{align*}
For every 
$(x,t) \in R_0 \cap \{ M^\sharp f(x,t) > \frac{1}{2A} (\lambda^r - \delta^r) \} $,
there exists $R_{(x,t)} \subset R_0$ such that $(x,t) \in R_{(x,t)}$ and 
\[
\dashint_{R_{(x,t)}^+(\gamma)} (f - c_{R_{(x,t)}})^r_+  + \dashint_{R_{(x,t)}^-(\gamma)} (f - c_{R_{(x,t)}})^r_- > \frac{1}{2A} (\lambda^r - \delta^r)
\]
By a similar argument to the Vitali covering theorem, 
we obtain
a countable collection $\{R_i\}_i$ of pairwise disjoint parabolic rectangles such that
\[
\big\{ (x,t) \in R_0: M^\sharp f(x,t) > \tfrac{1}{2A} (\lambda^r - \delta^r) \big\} 
\subset \bigcup_{l=1}^\infty 5 R_l .
\]
Then we have
\begin{align*}
&\bigl\lvert R_0 \cap \{ M^\sharp f(x,t) > \tfrac{1}{2A} (\lambda^r - \delta^r) \} \bigr\rvert \\
&\qquad \leq \sum_{l=1}^\infty \lvert 5 R_l \rvert =  \frac{5^{n+p} 2}{1-\gamma} \sum_{l=1}^\infty \lvert R_l^+(\gamma) \rvert \\
&\qquad \leq \frac{5^{n+p} 2}{1-\gamma} \frac{ 2^\frac{q}{r} A^\frac{q}{r}}{ (\lambda^r - \delta^r)^\frac{q}{r}}  \sum_{l=1}^\infty \lvert R_l^+(\gamma) \rvert \Biggl( \dashint_{R_l^+(\gamma)} (f - c_{R_l})^r_+  + \dashint_{R_l^-(\gamma)} (f - c_{R_l})^r_- \Biggr)^\frac{q}{r} \\
&\qquad \leq \frac{c_2 A^\frac{q}{r}}{2c_1} \frac{ \norm{f}^q}{(\lambda^r - \delta^r)^\frac{q}{r}} ,
\end{align*}
where 
$c_2 = 5^{n+p} 2^{\frac{q}{r}+2} c_1 / (1-\gamma)$.
On the other hand, if $k \notin \mathcal{K}$, we have
\[
\Bigl\lvert \bigcup_{j \in \mathcal{M}_k} \widetilde{S}^-_j \Bigr\rvert \leq \frac{1}{A} \lvert R^+_{k^-}(\gamma) \rvert .
\]
Summing over all indices $k \notin \mathcal{K}$ and applying \eqref{measure1} together with \eqref{measure2}, we conclude that
\begin{align*}
\sum_{k \notin \mathcal{K}} \sum_{j \in \mathcal{M}_k} \lvert \widetilde{S}^-_j \rvert &\leq \frac{1}{A} \sum_{k \notin \mathcal{K}} \lvert R^+_{k^-}(\gamma) \rvert \leq \frac{2}{A} \frac{1-\gamma}{1-\alpha} \sum_{k \notin \mathcal{K}} \lvert S^+_{k^-} \rvert \\
&\leq \frac{2}{A} \frac{1-\gamma}{1-\alpha} 2^{nm} \lceil 2^{pm} \rceil \sum_{k \notin \mathcal{K}} \lvert S^+_k \rvert 
\leq \frac{c_3}{Ac_1} \lvert S(\delta) \rvert ,
\end{align*}
where
\begin{align*}
2^{nm} \lceil 2^{pm} \rceil &\leq 2^{nm} 2^{pm+1} \leq 2^{1+(n+p) \bigl( \frac{1}{p} \log_2\bigl( \frac{3+\alpha}{2(\alpha- \gamma)}\bigr) + 1 \bigr) }  \\
&= 2^{1+n+p} \biggl( \frac{3+\alpha}{2(\alpha- \gamma)}\biggr)^{1 + \frac{n}{p}}
= \frac{1}{2c_1} \frac{1-\alpha}{1-\gamma} c_3 .
\end{align*}
By combining the cases $k\in \mathcal{K}$ and $k \notin \mathcal{K}$, we obtain
\begin{align*}
c_1 \sum_{k} \sum_{j \in \mathcal{M}_k} \lvert \widetilde{S}^-_j \rvert &= c_1 \sum_{k \in \mathcal{K}} \sum_{j \in \mathcal{M}_k} \lvert \widetilde{S}^-_j \rvert + c_1 \sum_{k \notin \mathcal{K}} \sum_{j \in \mathcal{M}_k} \lvert \widetilde{S}^-_j \rvert \\
&\leq \frac{c_2 A^\frac{q}{r}}{2} \frac{ \norm{f}^q}{(\lambda^r - \delta^r)^\frac{q}{r}} + \frac{c_3}{A} \lvert S(\delta) \rvert .
\end{align*}

It is left to consider the case $j \notin \mathcal{M}_k$.
Using \eqref{measure2}, we have
\begin{align*}
\lambda^r - \delta^r < 2 \dashint_{\widetilde{S}^-_j} (f - c_{R_j})^r_- \leq 4 \frac{1-\gamma}{1-\alpha} \dashint_{R^-_j(\gamma)} (f - c_{R_j})^r_- \leq 4 \frac{1-\gamma}{1-\alpha} M^\sharp f(x,t)
\end{align*}
for every $(x,t) \in R_{j}$ and thus for every $(x,t) \in \bigcup_{k} \bigcup_{j\notin \mathcal{M}_k} R_{j}$.
Then arguing as before, we obtain
\begin{align*}
\Big\lvert \bigcup_{k} \bigcup_{j\notin \mathcal{M}_k} \widetilde{S}^-_j \Big\rvert &\leq \Big\lvert \bigcup_{k} \bigcup_{j\notin \mathcal{M}_k} R_j \Big\rvert 
\leq \bigl\lvert R_0 \cap \{ M^\sharp f > \tfrac{1}{4} \tfrac{1-\alpha}{1-\gamma} (\lambda^r - \delta^r) \} \bigr\rvert \\
&\leq \frac{c_4}{2c_1} \frac{ \norm{f}^q}{(\lambda^r - \delta^r)^\frac{q}{r}} ,
\end{align*}
where 
$c_4 = c_2 ( 2(1-\gamma)/(1-\alpha))^{\frac{q}{r}} $.
By using~\eqref{start} and combining the cases $j \in \mathcal{M}_k$ and $j \notin \mathcal{M}_k$, we conclude that
\begin{align*}
\lvert S(\lambda) \rvert &\leq c_1 \sum_j \lvert \widetilde{S}^-_j \rvert = c_1 \sum_{k} \sum_{j \in \mathcal{M}_k} \lvert \widetilde{S}^-_j \rvert + c_1 \sum_{k} \sum_{j \notin \mathcal{M}_k} \lvert \widetilde{S}^-_j \rvert \\
&\leq A^\frac{q}{r} c_4 \frac{ \norm{f}^q}{(\lambda^r - \delta^r)^\frac{q}{r}} + \frac{c_3}{A} \lvert S(\delta) \rvert .
\end{align*}
By choosing $A = 2^{\frac{q}{r}+1} c_3$ and replacing $\lambda$ and $\delta$ by $2^\frac{1}{r} \lambda $ and $\lambda$, respectively, we have
\begin{equation}
\label{iter}
\lvert S( 2^\frac{1}{r} \lambda ) \rvert \leq A^\frac{q}{r} c_4 \frac{ \norm{f}^q}{ \lambda^q} +  \frac{c_3}{A} \lvert S(\lambda) \rvert .
\end{equation}
Let 
\[
\lambda_0 = \frac{ \norm{f}}{\lvert R_0^+(\alpha) \rvert^\frac{1}{q}} .
\]
For $0 < \lambda \leq \lambda_0$, we have
\[
\lvert R^{+}_0(\alpha) \cap \{ f(x,t)_+ > \lambda \} \rvert 
\leq \lvert R_0^+(\alpha) \rvert = \frac{\norm{f}^q}{\lambda_0^q} \leq \frac{\norm{f}^q}{\lambda^q} .
\]
Assume then that $\lambda > \lambda_0$.
There exists an integer $N \in \mathbb{N}$ such that $2^\frac{N}{r} \lambda_0 \leq \lambda < 2^\frac{N+1}{r} \lambda_0$.
We claim that
\[
\lvert S(2^\frac{K}{r} \lambda_0) \rvert \leq c_5 \frac{\norm{f}^q}{(2^{\frac{K}{r}} \lambda_0)^q}
\]
for every $ K=0,1,\dots,N$, where 
$ c_5 = 2^{\frac{q}{r}+1} A^\frac{q}{r} c_4 $.
We prove the claim by induction. 
First, note that the claim holds for $K=0$, because
\[
\lvert S(\lambda_0) \rvert \leq \lvert R_0^+(\alpha) \rvert = \frac{ \norm{f}^q}{\lambda_0^q} \leq c_5 \frac{ \norm{f}^q}{\lambda_0^q}.
\]
Then assume that the claim holds for $ K \in \{0,1,\dots,N-1\}$, that is, 
\[
\lvert S(2^\frac{K}{r} \lambda_0) \rvert \leq c_5 \frac{\norm{f}^q}{(2^\frac{K}{r} \lambda_0)^q} .
\]
We show that this implies the claim for $K+1$.
By using \eqref{iter} for $2^\frac{K}{r} \lambda_0$ we observe that
\begin{align*}
\lvert S( 2^\frac{K+1}{r} \lambda_0 ) \rvert &\leq A^\frac{q}{r} c_4 \frac{ \norm{f}^q}{ (2^\frac{K}{r} \lambda_0)^q} + \frac{c_3}{A} \lvert S(2^\frac{K}{r} \lambda_0) \rvert \leq A^\frac{q}{r} c_4  \frac{ \norm{f}^q}{ (2^\frac{K}{r} \lambda_0)^q} + \frac{c_3}{A} c_5 \frac{\norm{f}^q}{(2^\frac{K}{r} \lambda_0)^q} \\
&= 2^\frac{q}{r} \biggl( A^\frac{q}{r} c_4 + \frac{c_3}{A} c_5 \biggr) \frac{\norm{f}^q}{(2^\frac{K+1}{r} \lambda_0)^q} = c_5 \frac{\norm{f}^q}{(2^\frac{K+1}{r} \lambda_0)^q} .
\end{align*}
Therefore, the claim holds for $K+1$.

If $(x,t) \in S^+_0 \setminus S(\lambda)$, then there exists a sequence $\{S^+_l\}_{l\in\mathbb N}$
of subrectangles containing $(x,t)$ such that $c_{R_l} \leq \lambda $ and $\lvert S^+_l \rvert \to 0$ as $l \to \infty$.
Then
by \eqref{measure2} and Lemma~\ref{PJNp_constant}
it holds that
\begin{align*}
\dashint_{S^+_l} f^r_+ &\leq \dashint_{S^+_l} (f-c_{R_l})^r_+ + \lambda^r \leq \dashint_{S^+_l} (f-c_{R_l})^r_+ + \dashint_{S^-_l} (f-c_{R_l})^r_- + \lambda^r \\
&\leq
2\frac{1-\gamma}{1-\alpha} \inf_{c \in \mathbb{R}} \Biggl( \dashint_{R_l^+(\gamma)} (f-c)_+^r + \dashint_{R_l^-(\gamma)} (f-c)_-^r \Biggr) + \lambda^r \\
&\leq
2\frac{1-\gamma}{1-\alpha} \Biggl( \dashint_{R_l^+(\gamma)} (f-f(x,t))_+^r + \dashint_{R_l^-(\gamma)} (f-f(x,t))_-^r \Biggr) + \lambda^r .
\end{align*}
The Lebesgue differentiation theorem~\cite[Lemma~2.3]{KinnunenMyyryYang2022}
implies that
$f(x,t)^r_+ \leq \lambda^r$
for almost every $(x,t) \in S^+_0 \setminus S(\lambda)$.
It follows that
\[
\{ (x,t) \in S^+_0 :  f(x,t)_+ > \lambda \} \subset S(\lambda)
\]
up to a set of measure zero. 
We conclude that
\begin{align*}
\lvert S^+_0 \cap \{ f_+ > \lambda \} \rvert &\leq 
\lvert S(\lambda ) \rvert \leq \lvert S(2^\frac{N}{r} \lambda_0) \rvert \leq c_5 \frac{\norm{f}^q}{(2^{\frac{N}{r}} \lambda_0)^q} \\
&= 2^\frac{q}{r} c_5 \frac{\norm{f}^q}{(2^{\frac{N+1}{r}} \lambda_0)^q} \leq C \frac{\norm{f}^q}{\lambda^q} ,
\end{align*}
where $C = 2^\frac{q}{r} c_5$.
This completes the proof.
\end{proof}

\section{Chaining arguments and the time lag}

Weak type estimates in Theorem~\ref{local_pJN} together with chaining arguments
enable us to change the time lag in the parabolic John--Nirenberg inequality.

\begin{theorem}
\label{global_pJN}
Let $R \subset \mathbb{R}^{n+1}$ be a parabolic rectangle, $0 < \gamma <1$, $-1 < \rho \leq \gamma$, $-\rho < \sigma \leq \gamma$,
$1<q<\infty$
and $0< r \leq 1$.
Assume that $f \in PJN_{q,\gamma,r}^{+}(R)$.
Then there exist constants $c \in \mathbb{R}$
and $C=C(n,p,q,r,\gamma,\rho,\sigma)$
such that
\[
\lvert R^{+}(\rho) \cap \{ (f-c)_+ > \lambda \} \rvert 
\leq C \frac{\norm{f}_{PJN_{q,\gamma,r}^{+}(R)}^q}{\lambda^q}
\]
and
\[
\lvert R^{-}(\sigma) \cap \{ (f-c)_- > \lambda \} \rvert 
\leq C \frac{\norm{f}_{PJN_{q,\gamma,r}^{+}(R)}^q}{\lambda^q} 
\]
for every $\lambda>0$.
\end{theorem}

\begin{proof}

The beginning of the proof is similar to~\cite[Theorem~4.1]{KinnunenMyyryYang2022}.
Let $R_0 = R$
and $\norm{f} = \norm{f}_{PJN_{q,\gamma,r}^{+}(R_0)}$.
Without loss of generality, we may assume that the center of $R_0$ is the origin.
Since $f \in PJN_{q,\gamma,r}^{+}(R_0)$, Theorem~\ref{local_pJN} holds for any parabolic subrectangle of $R_0$ and for any $\gamma < \alpha <1$.
Let  $m$ be the smallest integer with
\[
m \geq \log_2 \biggl( \frac{1+\alpha}{1-\alpha} \biggr) + \frac{1}{p-1} \biggl( 2 + \log_2 \frac{1+\alpha}{\rho+\sigma} \biggr) + 2 .
\]
Then there exists $0 \leq \varepsilon < 1$ such that
\[ 
m = \log_2 \biggl( \frac{1+\alpha}{1-\alpha} \biggr) +  \frac{1}{p-1} \biggl( 2 + \log_2 \frac{1+\alpha}{\rho+\sigma} \biggr) + 2 + \varepsilon .
\]
We partition $R^+_0(\rho) = Q(0, L) \times (\rho L^p, L^p)$ by dividing each of its spatial edges into $2^m$ equally long intervals and the time interval into 
$\lceil (1-\rho)2^{mp}/(1-\alpha)\rceil$ equally long intervals.
Denote the obtained rectangles by $U^+_{i,j}$ with $i \in \{1,\dots,2^{mn}\}$ and  $j \in \{1,\dots,\lceil (1-\rho)2^{mp}/(1-\alpha)\rceil\}$.
The spatial edge length of $U^+_{i,j}$ is $l = l_x(U^+_{i,j}) =L/2^m$
and the time length is
\[
l_t(U^+_{i,j}) = \frac{(1-\rho)L^p}{\lceil (1-\rho)2^{mp}/(1-\alpha)\rceil} .
\]
For every $U^+_{i,j}$, there exists a unique rectangle $R_{i,j}$ that has the same top as $U^+_{i,j}$.
Our aim is to construct a chain from each $U^+_{i,j}$ to a central rectangle which is of the same form as $R_{i,j}$ and is contained in $R_0$. 
This central rectangle will be specified later.
First, we construct a chain with respect to the spatial variable.
Fix $U^+_{i,j}$.
Let $P_0 = R_{i,j}$ and
\[
P_0^+ = R^+_{i,j}(\alpha) = Q_i \times (t_j - (1-\alpha)l^p, t_j).
\]
We construct a chain of cubes from $Q_i$ to the central cube $Q(0,l)$.
Let $Q'_0 = Q_i = Q(x_i, l)$ and set
\[
Q'_k = Q'_{k-1} - \frac{x_i}{\abs{x_i}} \frac{\theta l}{2}
\quad \text{for every} \ k \in \{0,\dots,N_i\} ,
\]
where $1 \leq \theta \leq \sqrt{n}$ depends on the angle between $x_i$ and the spatial axes and is chosen such that the center of $Q'_k$ is on the boundary of $Q'_{k-1}$.
We have
\[
\frac{1}{2^n} \leq \frac{\lvert Q'_k \cap Q'_{k-1} \rvert}{\lvert Q'_k \rvert} \leq \frac{1}{2}
\quad \text{for every} \ k \in \{0,\dots,N_i\} ,
\]
and $\abs{x_i} = \frac{\theta}{2} (L - bl)$,
where $b \in \{1, \dots, 2^m\}$ depends on the distance of $Q_i$ to the center of $Q_0 = Q(0,L)$.
The number of cubes in the spatial chain $\{Q'_k\}_{k=0}^{N_i}$ is
\[
N_i + 1 = \frac{\abs{x_i}}{\frac{\theta}{2} l} + 1 = \frac{L}{l} - b + 1 .
\]

Next, we also take the time variable into consideration in the construction of the chain.
Let
\[
P^+_k = Q'_k \times ( t_j - (1-\alpha)l^p - k(1+\alpha)l^p, t_j - k(1+\alpha)l^p )
\]
and $P^-_k = P^+_k - (0, (1+\alpha)l^p)$,
for every $k \in \{ 0, \dots, N_i \}$,
be the upper and the lower parts of a parabolic rectangle respectively.
These will form a chain of parabolic rectangles from $U^+_{i,j}$ to the eventual central rectangle.
Observe that every rectangle $P_{N_i}$ coincides spatially for all pairs $(i,j)$.
Consider $j=1$ and such $i$ that the boundary of $Q_i$ intersects the boundary of $Q_0$.
For such a cube $Q_i$, we have $b=1$, and thus $N = N_i = \frac{L}{l} - 1$.
In the time variable, we travel from $t_1$ the distance
\[
(N+1)(1+\alpha)l^p + (1-\alpha)l^p = (1+\alpha) L l^{p-1} + (1-\alpha) l^p .
\]
We show that the lower part of the final rectangle $P^-_N$ is contained in $R_0$.
To this end, we subtract the time length of $U^+_{i,1}$ from the distance above and observe that it is less than half of the time length of $R_0 \setminus(R_0^+(\rho) \cup R_0^-(\sigma))$.
This follows from the computation
\begin{align*}
&(1+\alpha) L l^{p-1} + (1-\alpha) l^p - \frac{(1-\rho)L^p}{\lceil(1-\rho)2^{mp}/(1-\alpha)\rceil}\\
&\qquad= \biggl( \frac{1+\alpha}{2^{m(p-1)}} + \frac{1-\alpha}{2^{mp}} - \frac{1-\rho}{\lceil (1-\rho)2^{mp}/(1-\alpha) \rceil} \biggr) L^p \\
&\qquad\leq \Biggl( \frac{1+\alpha}{2^{m(p-1)}} + \frac{1-\alpha}{2^{mp}} - \frac{1-\rho}{2\frac{(1-\rho)2^{mp}}{1-\alpha}} \Biggr) L^p  \\
&\qquad= \biggl( \frac{1+\alpha}{2^{m(p-1)}} + \frac{1-\alpha}{2^{mp+1}} \biggr) L^p 
\leq 2 \frac{1+\alpha}{2^{m(p-1)}} L^p \leq \frac{\rho+\sigma}{2} L^p ,
\end{align*}
because 
\[
m \geq \frac{1}{p-1} \biggl( 2 + \log_2 \frac{1+\alpha}{\rho+\sigma} \biggr).
\]
This implies that $P^-_N \subset R_0^+(\rho - (\rho+\sigma)/2)$.
Denote this rectangle $P_N$ by $\mathfrak{R} = \mathfrak{R}_\rho$.
This is the central rectangle where all chains will eventually end.

Let $j=1$ and assume that $i$ is arbitrary. We extend the chain $\{P_k\}_{k=0}^{N_i}$ by $N - N_i$ rectangles into the negative time direction such that the final rectangle coincides with the central rectangle $\mathfrak{R}$.
More precisely, we consider $Q'_{k+1} = Q'_{N_i}$, 
\[
P^+_{k+1} = P^+_{k} - (0, (1+\alpha)l^p) 
\quad\text{and}\quad
P^-_{k+1} = P^+_{k+1} - (0, (1+\alpha)l^p)
\]
for $k \in \{N_i, \dots, N-1\}$. 
For every $j \in \{2,\dots,\lceil (1-\rho)2^{mp}/(1-\alpha)\rceil \}$, we consider a similar extension of the chain.
The final rectangles of the chains coincide for fixed $j$ and for every $i$.
Moreover, every chain is of the same length $N+1$, and it holds that
\[
\frac{1}{2^n} \leq \eta = \frac{\lvert P_k^+ \cap P_{k-1}^- \rvert}{\lvert P_k^+ \rvert} \leq 1 .
\]

Then we consider an index $j \in \{2,\dots,\lceil(1-\rho)2^{mp}/(1-\alpha)\rceil \}$ related to the time variable.
The time distance between the current ends of the chains for pairs $(i,j)$ and $(i,1)$ is
\[
(j-1) \frac{(1-\rho) L^p}{\lceil (1-\rho)2^{mp}/(1-\alpha)\rceil} .
\]
Our objective is to have the final rectangle of the continued chain for $(i,j)$ to coincide with the end of the chain for $(i,1)$, namely, with the central rectangle $\mathfrak{R}$.
To achieve this, we modify $2^{m-1}$ intersections of $P^+_k$ and $P^-_{k+1}$ by shifting $P_k$ and also add a chain of $M_j$ rectangles traveling to the negative time direction into the chain $\{P_k\}_{k=0}^{N}$.
We shift every $P_k, k \in \{ 1, \dots, 2^{m-1} \}$, by a $\beta_j$-portion of their temporal length more than the previous rectangle was shifted, that is, we move each $P_k$ into the negative time direction a distance of $k \beta_j (1-\alpha) l^p$.
The values of $M_j \in \mathbb{N}$ and $0\leq \beta_j <1$ will be chosen later.
In other words, for every $ k \in \{ 1, \dots, 2^{m-1} \}$, we modify the definitions of $P^+_k$ by
\[
P^+_k = Q'_k \times ( t_j - (1-\alpha)l^p - k( 1+\alpha + \beta_j (1-\alpha)) l^p, t_j - k( 1+\alpha + \beta_j (1-\alpha)) l^p) ,
\]
and then add $M_j$ rectangles defined by
\[
P^+_{k+1} = P^+_{k} - (0, (1+\alpha)l^p) 
\quad\text{and}\quad
P^-_{k+1} = P^+_{k+1} - (0, (1+\alpha)l^p) 
\]
for every $k \in \{N,\dots, N+M_j-1\}$.
Note that there exists $1 \leq \tau < 2$ such that 
\[
\tau \frac{(1-\rho)2^{mp}}{1-\alpha} 
=\left\lceil\frac{(1-\rho)2^{mp}}{(1-\alpha)}\right\rceil.
\]
We would like to find such $0\leq \beta_j <1$ and $M_j \in \mathbb{N}$ that 
\[
(j-1) \frac{(1-\rho) L^p}{\lceil (1-\rho)2^{mp}/(1-\alpha)\rceil} - M_j \frac{(1+\alpha) L^p}{2^{mp}} 
= 2^{m-1} \beta_j \frac{(1-\alpha) L^p}{2^{mp}} ,
\]
which is equivalent with
\[
(j-1)\tau^{-1} (1-\alpha) - M_j (1+\alpha) = 2^{m-1} \beta_j (1-\alpha) .
\]
With this choice all final rectangles coincide.
Choose $M_j \in \mathbb{N}$ such that
\[
M_j (1+\alpha) \leq (j-1) \tau^{-1} (1-\alpha) < (M_j + 1)(1+\alpha) ,
\]
that is,
\[
0 \leq \xi = (j-1) \tau^{-1} (1-\alpha) - M_j (1+\alpha) < 1+\alpha 
\]
and
\[
\frac{(j-1)(1-\alpha)}{2 (1+\alpha)} -1 \leq \frac{(j-1)(1-\alpha)}{\tau (1+\alpha)} -1 < M_j \leq \frac{(j-1)(1-\alpha)}{\tau (1+\alpha)} \leq \frac{(j-1)(1-\alpha)}{1+\alpha} .
\]
By choosing $0\leq \beta_j <1$ such that
\[
\xi = 2^{m-1} \beta_j (1-\alpha) = 2^{\frac{2}{p-1} + 1 + \varepsilon} \biggl( \frac{1+\alpha}{\rho+\sigma} \biggr)^\frac{1}{p-1} \beta_j (1+\alpha) ,
\]
we have
\[
\beta_j = 2^{-\frac{2}{p-1} - 1 - \varepsilon} \biggl( \frac{\rho+\sigma}{1+\alpha} \biggr)^\frac{1}{p-1} \frac{\xi}{1+\alpha} .
\]
Observe that $0\leq \beta_j \leq \frac{1}{2}$ for every $j$.
For measures of the intersections of the modified rectangles, it holds that
\[
\frac{1}{2^{n+1}} \leq \frac{\lvert P_k^+ \cap P_{k-1}^- \rvert}{\lvert P_k^+ \rvert} = \eta (1-\beta_j) \leq 1 
\]
for every $ k \in \{ 1, \dots, 2^{m-1} \}$, and thus
\[
\frac{1}{2^{n+1}} \leq \tilde{\eta}_j = \frac{\lvert P_k^+ \cap P_{k-1}^- \rvert}{\lvert P_k^+ \rvert} \leq 1 
\]
for every $k \in \{1,\dots, N+M_j\}$.
Fix $U^+_{i,j}$. 
We conclude that
\begin{align*}
(c_{R_{i,j}} - c_{\mathfrak{R}} )^q_+ &= (c_{P_0} - c_{P_{N+M_j}} )^q_+  \leq  \left( \sum_{k=1}^{N+M_j} (c_{P_{k-1}} - c_{P_{k}} )^r_+ \right)^\frac{q}{r} \\
&=  \left( \sum_{k=1}^{N+M_j} \dashint_{P_{k-1}^- \cap P_k^+} (c_{P_{k-1}} - c_{P_{k}} )^r_+ \right)^\frac{q}{r} \\
&\leq  \left( \sum_{k=1}^{N+M_j} \Biggl( \dashint_{P_{k-1}^- \cap P_k^+} (c_{P_{k-1}} - f )^r_+ + \dashint_{P_{k-1}^- \cap P_k^+} (f- c_{P_{k}} )^r_+  \Biggr) \right)^\frac{q}{r} \\
&\leq  \left( \sum_{k=1}^{N+M_j} \frac{1}{\tilde{\eta}_j} \Biggl( \dashint_{P_{k-1}^-} (f - c_{P_{k-1}})^r_- + \dashint_{P_k^+} (f- c_{P_{k}} )^r_+  \Biggr) \right)^\frac{q}{r} \\
&\leq 2^{(n+1)\frac{q}{r}}  \left(  \sum_{k=0}^{N+M_j} \Biggl( \dashint_{P_{k}^-} (f - c_{P_{k}} )^r_- + \dashint_{P_k^+} (f- c_{P_{k}} )^r_+ \Biggr) \right)^\frac{q}{r} \\
&\leq 2^{(n+1)\frac{q}{r}} (N+1+M_j)^{\frac{q}{r} -1} \sum_{k=0}^{N+M_j} \Biggl( \dashint_{P_{k}^-} (f - c_{P_{k}} )^r_- + \dashint_{P_k^+} (f- c_{P_{k}} )^r_+ \Biggr)^\frac{q}{r} \\
&= 2^{(n+1)\frac{q}{r}} (N+1+M_j)^{\frac{q}{r} -1} \frac{1}{\lvert P_0^+ \rvert} \sum_{k=0}^{N+M_j} \lvert P_k^+ \rvert \Biggl( \dashint_{P_{k}^-} (f - c_{P_{k}} )^r_- + \dashint_{P_k^+} (f- c_{P_{k}} )^r_+ \Biggr)^\frac{q}{r} \\ 
&\leq 2^{(n+1)\frac{q}{r}} (N+1+M_j)^{\frac{q}{r} -1} 2 \frac{\norm{f}^q}{\lvert P_0^+ \rvert} ,
\end{align*}
where in the last inequality we used that $P_k$ are pairwise disjoint separately for even $k$ and for odd $k$.
We have
\begin{align*}
N+1+M_j &= \frac{L}{l}+M_j \leq 2^m + (j-1) \frac{1-\alpha}{1+\alpha} \\
&\leq 2^m + \frac{(1-\rho)2^{mp}}{1-\alpha} \frac{1-\alpha}{1+\alpha} 
\leq 2^m + 2^{mp+1} \leq 2^{mp+2} \\
&\leq 2^{\frac{2p}{p-1} + 3p +2} \biggl( \frac{1+\alpha}{\rho+\sigma} \biggr)^\frac{p}{p-1} \biggl( \frac{1+\alpha}{1-\alpha} \biggr)^p
\end{align*}
for every $j$.
Hence, we obtain
\[
(c_{R_{i,j}} - c_{\mathfrak{R}} )_+ 
\leq B \frac{\norm{f}}{\lvert P_0^+ \rvert^\frac{1}{q}} 
\quad\text{with}\quad
B = 2^{ \frac{1}{q} + (n+1)\frac{1}{r} } 
\biggl(
2^{\frac{2p}{p-1} + 3p +2 }
\biggl( \frac{1+\alpha}{\rho+\sigma} \biggr)^{\frac{p}{p-1}} \biggl( \frac{1+\alpha}{1-\alpha} \biggr)^p
\biggr)^{\frac{1}{r}-\frac{1}{q}} .
\]

We observe that
\begin{align*}
&\lvert R^{+}_0(\rho) \cap \{ (f-c_{\mathfrak{R}})_+ > \lambda \} \rvert 
\\
&\qquad \leq \sum_{i,j} \lvert R^+_{i,j}(\alpha) \cap \{ (f-c_{\mathfrak{R}})_+ > \lambda \} \rvert \\
&\qquad\leq \sum_{i,j} \lvert R^+_{i,j}(\alpha) \cap \{ (f-c_{R_{i,j}})_+ > \tfrac{\lambda}{2} \} \rvert 
 + \sum_{i,j} \lvert R^+_{i,j}(\alpha) \cap \{ (c_{R_{i,j}}-c_{\mathfrak{R}})_+ > \tfrac{\lambda}{2} \} \rvert .
\end{align*}
The first sum of the right-hand side can be estimated by Theorem~\ref{local_pJN} as follows
\begin{align*}
\sum_{i,j} \lvert R^+_{i,j}(\alpha) \cap \{ (f-c_{R_{i,j}})_+ > \tfrac{\lambda}{2} \} \rvert 
&\leq \sum_{i,j} 2^q C \frac{\norm{f}^q_{PJN^+_{q,\gamma,r}(R_{i,j})}}{\lambda^q} \\
&\leq 2^q C \left\lceil \frac{l_t(R_{i,j})}{l_t(U^+_{i,j})} \right\rceil \frac{\norm{f}^q_{PJN^+_{q,\gamma,r}(R_0)}}{\lambda^q} \\
&\leq 2^q C \left\lceil \frac{4}{1-\alpha} \right\rceil \frac{\norm{f}^q_{PJN^+_{q,\gamma,r}(R_0)}}{\lambda^q} \\
&\leq  \frac{2^{q+3} C}{1-\alpha} \frac{\norm{f}^q_{PJN^+_{q,\gamma,r}(R_0)}}{\lambda^q} .
\end{align*}
To estimate the second sum above, assume that $\lambda \geq 2 B \norm{f} / \lvert P_0^+ \rvert^\frac{1}{q}$.
This implies that
\begin{align*}
\lvert R^+_{i,j}(\alpha) \cap \{ (c_{R_{i,j}}-c_{\mathfrak{R}})_+ > \tfrac{\lambda}{2} \} \rvert 
&\leq \lvert R^+_{i,j}(\alpha) \cap \{  B \norm{f} / \lvert P_0^+ \rvert^\frac{1}{q} > \tfrac{\lambda}{2} \} \rvert = 0
\end{align*}
for every $i,j$.
Thus,
\[
\lvert R^{+}_0(\rho) \cap \{ (f-c_{\mathfrak{R}})_+ > \lambda \} \rvert 
\leq  \frac{2^{q+3} C}{1-\alpha} \frac{\norm{f}^q}{\lambda^q}
\]
for every $\lambda \geq 2 B \norm{f} / \lvert P_0^+ \rvert^\frac{1}{q}$.
For $0<\lambda < 2 B \norm{f} / \lvert P_0^+ \rvert^\frac{1}{q}$, we have
\begin{align*}
& \lvert R^{+}_0(\rho) \cap \{ (f-c_{\mathfrak{R}})_+ > \lambda \} \rvert 
\\
&\qquad \leq \lvert R^+_0(\rho) \rvert = 2^{(n+p)m} \frac{1-\rho}{1-\alpha} \lvert P^+_0 \rvert \\
&\qquad \leq 2^{(n+p)(\frac{2}{p-1} + 3 )} \biggl( \frac{1+\alpha}{\rho+\sigma} \biggr)^\frac{n+p}{p-1} \biggl( \frac{1+\alpha}{1-\alpha} \biggr)^{n+p}  \frac{1-\rho}{1-\alpha} 2^q B^q \frac{\norm{f}^q}{\lambda^q} .
\end{align*}

We can apply a similar chaining argument in the reverse time direction for $R_0^-(\sigma)$ with the exception that we also extend (and modify if needed) every chain such that the central rectangle $\mathfrak{R}_\sigma$ coincides with $\mathfrak{R}=\mathfrak{R}_\rho$.
A rough upper bound for the number of rectangles needed for the additional extension is given by
\[
\left\lceil \frac{(\rho+\sigma)L^p}{(1+\alpha)l^p} \right\rceil = \left\lceil \frac{\rho+\sigma}{1+\alpha} 2^{mp} \right\rceil \leq 2^{mp+1} .
\]
Thus, the constant $B$ above is
$2^{\frac{1}{r} - \frac{1}{q} }$
times larger in this case. 
This proves the second inequality of the theorem.
\end{proof}

The next corollary of Theorem~\ref{global_pJN} tells that the spaces $PJN_{q,\gamma,r}^{+}$ coincide for every $0 < \gamma <1$ and $0<r<q$.

\begin{corollary}
\label{equivPJNq}
Let $\Omega_T \subset \mathbb{R}^{n+1}$ be a space-time cylinder,
$0<\rho \leq \gamma < 1$, $1 < q < \infty$ and $0<r \leq s<q <\infty$.
Then there exist constants $c_1=c_1(n,p,q,r,s,\gamma,\rho)$ and $c_2=c_2(n,p,q,r,s,\gamma,\rho)$ such that 
\[
c_1 \norm{f}_{PJN_{q,\gamma,r}^{+}(\Omega_T)} \leq \norm{f}_{PJN_{q,\rho,s}^{+}(\Omega_T)} \leq c_2 \norm{f}_{PJN_{q,\gamma,r}^{+}(\Omega_T)} .
\]
\end{corollary}
\begin{proof}

Let $\{R_i\}_{i\in\mathbb N}$ be a collection of pairwise disjoint parabolic subrectangles of $\Omega_T$.
By H\"older's inequality, we have
\begin{align*}
&\Biggl( \dashint_{R_i^+(\gamma)} (f-c_{R_i})_+^r + \dashint_{R_i^-(\gamma)} (f-c_{R_i})_-^r  \Biggr)^\frac{1}{r} \\
&\qquad \ \leq \max\{1, 2^{\frac{1}{r}-1} \} \Biggl( \biggl( \dashint_{R_i^+(\gamma)} (f-c_{R_i})_+^r  \biggr)^\frac{1}{r} + \biggl( \dashint_{R_i^-(\gamma)} (f-c_{R_i})_-^r \biggr)^\frac{1}{r} \Biggr) \\
&\qquad \ \leq \max\{1, 2^{\frac{1}{r}-1} \} \Biggl( \biggl( \dashint_{R_i^+(\gamma)} (f-c_{R_i})_+^s \biggr)^\frac{1}{s} + \biggl( \dashint_{R_i^-(\gamma)} (f-c_{R_i})_-^s \biggr)^\frac{1}{s} \Biggr) \\
&\qquad \ \leq c_0 \Biggl( \dashint_{R_i^+(\gamma)} (f-c_{R_i})_+^s + \dashint_{R_i^-(\gamma)} (f-c_{R_i})_-^s \Biggr)^\frac{1}{s} ,
\end{align*}
where $c_0 = \max\{1, 2^{\frac{1}{r}-1} \} \max\{1, 2^{1 -\frac{1}{s}} \}$.
We observe that
\begin{align*}
& \lvert R_i^+(\gamma) \rvert \Biggl( \dashint_{R_i^+(\gamma)} (f-c_{R_i})_+^s + \dashint_{R_i^-(\gamma)} (f-c_{R_i})_-^s \Biggr)^\frac{q}{s} \\
& \qquad \qquad \leq \biggl( \frac{1-\rho}{1-\gamma} \biggr)^{\frac{q}{s}-1} \lvert R_i^+(\rho) \rvert \Biggl( \dashint_{R_i^+(\rho)} (f-c_{R_i})_+^s + \dashint_{R_i^-(\rho)} (f-c_{R_i})_-^s \Biggr)^\frac{q}{s} .
\end{align*}
By summing over $i\in\mathbb{N}$ and taking supremum over all collections of pairwise disjoint parabolic rectangles,
we arrive at
\[
\norm{f}_{PJN_{q,\gamma,r}^{+}(\Omega_T)} \leq c_0 \biggl( \frac{1-\rho}{1-\gamma} \biggr)^{\frac{1}{s}-\frac{1}{q}} \norm{f}_{PJN_{q,\rho,s}^{+}(\Omega_T)}.
\]

To show the second inequality, we make the restriction $0 < r \leq 1$ so that we can apply Theorem~\ref{global_pJN}. This is not an issue because after establishing the second inequality for $0 < r \leq 1$ we can use the first inequality to obtain the whole range $0 < r \leq s$.
Cavalieri's principle and Theorem~\ref{global_pJN} imply that
\begin{align*}
\int_{R_i^+(\rho)} (f-c)_+^s &= s \int_0^\infty \lambda^{s-1} \lvert R_i^{+}(\rho) \cap \{ (f-c)_+ > \lambda \} \rvert \dla \\
&\leq s \int_{|R_i^{+}(\rho)|^{-\frac{1}{q}} \norm{f}_{PJN_{q,\gamma,r}^{+}(R_i)} }^\infty C \lambda^{s-q-1} \norm{f}_{PJN_{q,\gamma,r}^{+}(R_i)}^q \dla \\ &\qquad + s \int_0^{|R_i^{+}(\rho)|^{-\frac{1}{q}} \norm{f}_{PJN_{q,\gamma,r}^{+}(R_i)}} \lambda^{s-1} |R_i^{+}(\rho)| \dla \\
&= \frac{Cs}{q-s} |R_i^{+}(\rho)|^{1-\frac{s}{q}} \norm{f}_{PJN_{q,\gamma,r}^{+}(R_i)}^s + |R_i^{+}(\rho)|^{1-\frac{s}{q}} \norm{f}_{PJN_{q,\gamma,r}^{+}(R_i)}^s \\
&\leq \frac{Cq}{q-s} |R_i^{+}(\rho)|^{1-\frac{s}{q}} \norm{f}_{PJN_{q,\gamma,r}^{+}(R_i)}^s .
\end{align*}
Similarly, we obtain
\[
\int_{R_i^-(\rho)} (f-c)_-^s \leq \frac{Cq}{q-s} |R_i^{+}(\rho)|^{1-\frac{s}{q}} \norm{f}_{PJN_{q,\gamma,r}^{+}(R_i)}^s .
\]
These estimates imply that
\begin{align*}
\sum_{i=1}^\infty \lvert R_i^+(\rho) \rvert \inf_{c \in \mathbb{R}} \Biggl( \dashint_{R_i^+(\rho)} (f-c)_+^s + \dashint_{R_i^-(\rho)} (f-c)_-^s \Biggr)^\frac{q}{s}
&\leq \biggl( \frac{2Cq}{q-s} \biggr)^\frac{q}{s} \sum_{i=1}^\infty \norm{f}_{PJN_{q,\gamma,r}^{+}(R_i)}^q \\
&\leq \biggl( \frac{2Cq}{q-s} \biggr)^\frac{q}{s} \norm{f}_{PJN_{q,\gamma,r}^{+}(\Omega_T)}^q .
\end{align*}
Thus, we conclude that
\[
\norm{f}_{PJN_{q,\rho,s}^{+}(\Omega_T)} \leq \biggl( \frac{2Cq}{q-s} \biggr)^\frac{1}{s} \norm{f}_{PJN_{q,\gamma,r}^{+}(\Omega_T)} .
\qedhere
\]
\end{proof}


\begin{thebibliography}{10}

\bibitem{aimar}
H.~Aimar, \emph{Elliptic and parabolic {BMO} and {H}arnack's inequality}, Trans. Amer. Math. Soc. \textbf{306} (1988), no.~1, 265--276.

\bibitem{berkovits2012}
L.~Berkovits, \emph{Parabolic {J}ohn-{N}irenberg spaces}, J. Funct. Spaces Appl. (2012), Art. ID 901917, 9.

\bibitem{dafni2018}
G.~Dafni, T.~Hyt\"{o}nen, R.~Korte, and H.~Yue, \emph{The space {$JN_p$}: nontriviality and duality}, J. Funct. Anal. \textbf{275} (2018), no.~3, 577--603.

\bibitem{milman2021}
\'{O}. Dom\'{\i}nguez and M.~Milman, \emph{Sparse {B}rudnyi and {J}ohn-{N}irenberg spaces}, C. R. Math. Acad. Sci. Paris \textbf{359} (2021), 1059--1069.

\bibitem{fabesgarofalo}
E.~B. Fabes and N.~Garofalo, \emph{Parabolic {B}.{M}.{O}. and {H}arnack's inequality}, Proc. Amer. Math. Soc. \textbf{95} (1985), no.~1, 63--69.

\bibitem{john_original}
F.~John and L.~Nirenberg, \emph{On functions of bounded mean oscillation}, Comm. Pure Appl. Math. \textbf{14} (1961), 415--426.

\bibitem{KinnunenMyyry2023}
J.~Kinnunen and K.~Myyryl{\"a}inen, \emph{Characterizations of parabolic {M}uckenhoupt classes}, preprint (2023).

\bibitem{kinnunenmyyry2023a}
J.~Kinnunen and K.~Myyryl\"{a}inen, \emph{Dyadic {J}ohn-{N}irenberg space}, Proc. Roy. Soc. Edinburgh Sect. A \textbf{153} (2023), no.~1, 1--18.

\bibitem{KinnunenMyyryYang2022}
J.~Kinnunen, K.~Myyryl{\"a}inen, and D.~Yang, \emph{John--{N}irenberg inequalities for parabolic {BMO}}, Math. Ann. (2022), \url{https://doi.org/10.1007/s00208-022-02480-y}.

\bibitem{KinnunenMyyryYangZhu2022}
J.~Kinnunen, K.~Myyryl{\"a}inen, D.~Yang, and C.~Zhu, \emph{Parabolic {M}uckenhoupt weights with time lag on spaces of homogeneous type with monotone geodesic property}, Potential Analysis (2023), \url{https://doi.org/10.1007/s11118-023-10098-1}.

\bibitem{kinnunenSaariParabolicWeighted}
J.~Kinnunen and O.~Saari, \emph{Parabolic weighted norm inequalities and partial differential equations}, Anal. PDE \textbf{9} (2016), no.~7, 1711--1736.

\bibitem{kortetakala2023}
R.~Korte and T.~Takala, \emph{The {J}ohn-{N}irenberg space: {E}quality of the vanishing subspaces {$VJN_p$} and {$CJN_p$}}, arXiv:2303.12064 (2023).

\bibitem{marolasaari}
N.~Marola and O.~Saari, \emph{Local to global results for spaces of {BMO} type}, Math. Z. \textbf{282} (2016), no.~1-2, 473--484.

\bibitem{martin1994}
F.~J. Mart\'{\i}n-Reyes and A.~de~la Torre, \emph{One-sided {BMO} spaces}, J. London Math. Soc. (2) \textbf{49} (1994), no.~3, 529--542.

\bibitem{molla2022}
MD~N. Molla, \emph{John--{N}irenberg space on {LCA} groups}, Anal. Math. Phys. \textbf{12} (2022), no.~5, Paper No. 119, 19.

\bibitem{moser1964}
J.~Moser, \emph{A {H}arnack inequality for parabolic differential equations}, Comm. Pure Appl. Math. \textbf{17} (1964), 101--134.

\bibitem{moser1967}
\bysame, \emph{Correction to: ``{A} {H}arnack inequality for parabolic differential equations''}, Comm. Pure Appl. Math. \textbf{20} (1967), 231--236.

\bibitem{myyry2022}
K.~Myyryl\"{a}inen, \emph{Median-type {J}ohn-{N}irenberg space in metric measure spaces}, J. Geom. Anal. \textbf{32} (2022), no.~4, Paper No. 131, 23.

\bibitem{localtoglobal}
O.~Saari, \emph{Parabolic {BMO} and global integrability of supersolutions to doubly nonlinear parabolic equations}, Rev. Mat. Iberoam. \textbf{32} (2016), no.~3, 1001--1018.

\bibitem{forward-in-time}
\bysame, \emph{Parabolic {BMO} and the forward-in-time maximal operator}, Ann. Mat. Pura Appl. (4) \textbf{197} (2018), no.~5, 1477--1497.

\bibitem{takala2022}
T.~Takala, \emph{Nontrivial examples of {$JN_p$} and {$VJN_p$} functions}, Math. Z. \textbf{302} (2022), no.~2, 1279--1305.

\bibitem{tao2019}
J.~Tao, D.~Yang, and W.~Yuan, \emph{John-{N}irenberg-{C}ampanato spaces}, Nonlinear Anal. \textbf{189} (2019), 111584, 36.

\bibitem{tao2021survey}
\bysame, \emph{A survey on function spaces of {J}ohn--{N}irenberg type}, Mathematics \textbf{9} (2021), no.~18, 2264.

\bibitem{tao2022}
\bysame, \emph{Vanishing {J}ohn-{N}irenberg spaces}, Adv. Calc. Var. \textbf{15} (2022), no.~4, 831--861.

\end{thebibliography}
\end{document}